\documentclass[12pt, a4paper, reqno, twoside, makeidx]{amsart}
\usepackage[margin=2.7cm, marginpar=3cm]{geometry}

\usepackage{verbatim}
\usepackage{amsmath}
\usepackage{amsfonts}
\usepackage{amssymb}
\usepackage{graphicx}
\usepackage{color}
\usepackage{empheq}
\usepackage[all]{xy}
\usepackage{tikz-cd}
\usepackage{tikz}
\usetikzlibrary{calc}
\usepackage{amssymb}
\usepackage{amsmath, amsthm, amssymb, amsfonts, amscd}
\usepackage[english]{babel}
\usepackage{url}
\usepackage{hyperref}
\usepackage{mathtools}
\usepackage{bm}
\usepackage{dirtytalk}
\usepackage[percent]{overpic}
\usepackage{CJKutf8}
\usepackage{blkarray}
\usepackage{cancel}
\usepackage{thmtools}
\usepackage{thm-restate}

\usepackage{caption}
\theoremstyle{remark}
\newtheorem{remk}{Remark}[section]
\theoremstyle{plain}
\newtheorem{thm}{Theorem}[section]
\newtheorem{lem}[thm]{Lemma}
\newtheorem{prop}[thm]{Proposition}
\newtheorem{defi}[thm]{Definition}
\newtheorem{cor}[thm]{Corollary}

\theoremstyle{definition}
\newtheorem{exa}[thm]{Example}

\newcommand{\Z}{\mathbb{Z}}

\newcommand{\R}{\mathbb{R}}
\newcommand{\C}{\mathbb{C}}

\author{Antonio Alfieri and Keegan Boyle}

\title[Signature of strongly invertible knots ]{Strongly invertible knots, invariant surfaces,\\ and the Atiyah-Singer signature theorem}
\begin{document}

\begin{abstract}
We use the $G$-signature theorem to define an invariant of strongly invertible knots analogous to the knot signature.
\end{abstract}

\maketitle

\thispagestyle{empty}

\section{Introduction}
A \emph{symmetric knot} $(K,\rho)$ is a smooth knot $K \subset S^3$ along with a finite order diffeomorphism $\rho:S^3 \to S^3$ which leaves $K$ invariant. If $\rho$ preserves the orientation on $S^3$ and reverses the orientation on $K$ then we say that $(K,\rho)$ is \emph{strongly invertible}. Note that by Smith theory, the fixed set of a strong inversion is a circle $A\subset S^3$, which  must be unknotted \cite{MorganBass}, and intersect $K$ in two points. (If instead $A$ is disjoint from $K$ then $(K,\rho)$ is called \emph{periodic}.)

The three-dimensional topology of symmetric knots has been studied extensively using both classical techniques  \cite{MR292060,MR1102258}, and the modern methods of Khovanov and Floer homology \cite{MR4286365,MR4125758}. Their four-dimensional topology on the other hand is much less explored. Some work regarding the equivariant four-genus of periodic and strongly invertible knots has been done recently by Issa and the second author \cite{BoyleIssa}. Other work regarding the four-dimensional topology of symmetric knots has been done by Davis and Naik \cite{MR2216254}, and by Dai, Hedden and Mallick \cite{DaiHeddenMallick}, building on the methods of \cite{MR4126897} based on ideas from involutive Heegaard Floer homology \cite{MR3649355,MR4205781}.

\begin{figure}[h!]
\vspace{0.35cm}
\scalebox{.3}{\includegraphics{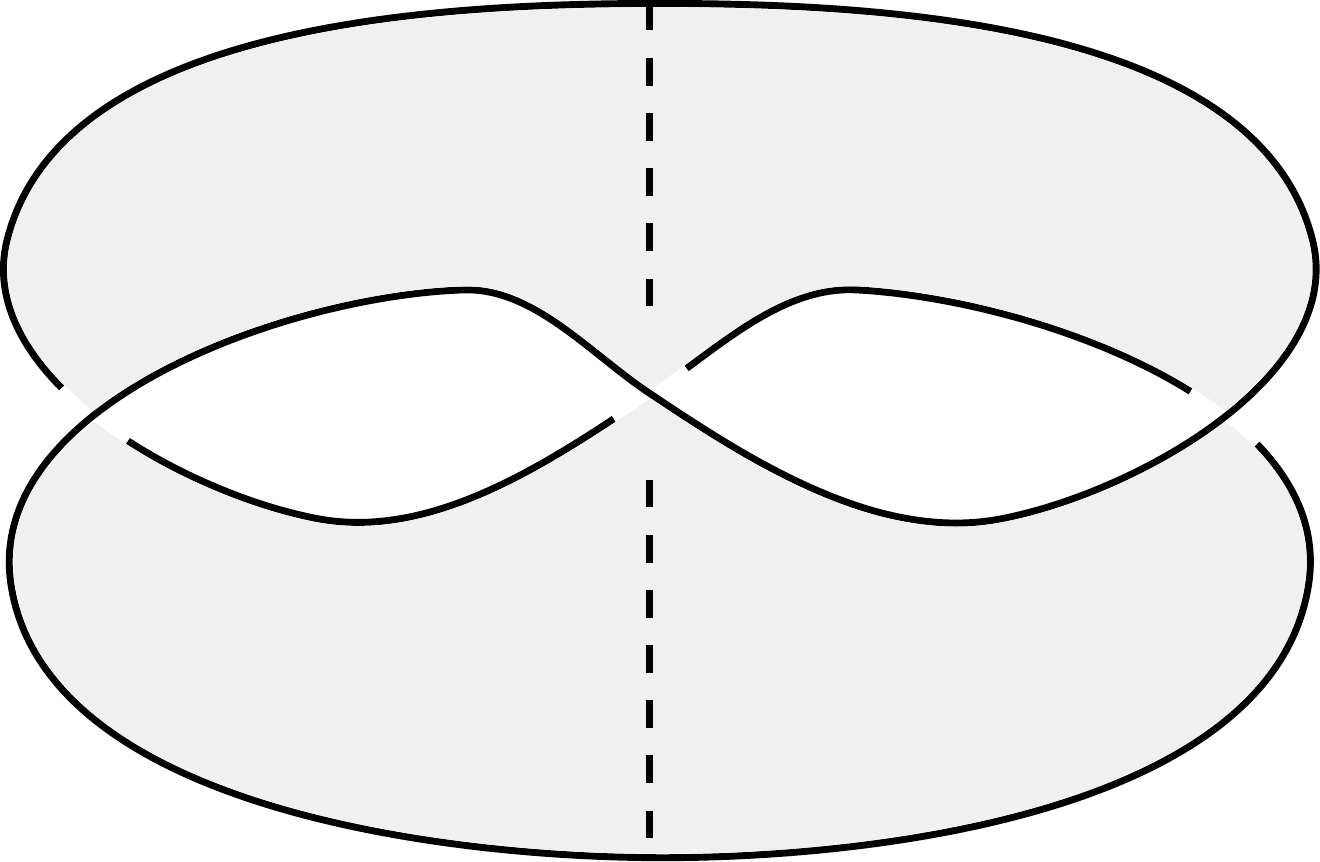}}
\caption{%An orientable surface with boundary the unknot. The surface is invariant under the unique strong inversion on the unknot, which is given by rotation around the dotted axis.
}
\label{fig:unknot}
\end{figure}

In \cite{MR143201} Trotter defines a bilinear pairing $\theta : H_1(F)\otimes H_1(F) \to \Z$ associated to  a Seifert surface $F\subset S^3$. If $K$ is a strongly invertible knot (or a periodic knot) then there is an invariant Seifert surface bounding $K$, and one can consider the involution $\rho_*: H_1(F) \to H_1(F)$ induced by the knot symmetry. In this case $H_1(F)$ splits into the direct sum $H_1(F)=E_+\oplus E_-$ of the $+1$ and $-1$ eigenspaces  of $\rho_*$. Taking the difference of the signatures 
\[
\widetilde{\sigma} = \sigma((\theta+\theta^T)|_{E_+}) - \sigma((\theta+\theta^T)|_{E_-})
\] 
we get a notion of equivariant signature. In \cite{BoyleIssa} the authors prove that in the case of periodic knots $\widetilde{\sigma}$ does not depend on the choice of the invariant Seifert surface,  and that it defines an equivariant concordance invariant \cite[Proposition 16]{BoyleIssa}. 

However, if $K$ is a strongly invertible knot then $\widetilde{\sigma}$ \emph{does} depend on the choice of the invariant Seifert surface. For example, from the invariant genus one Seifert surface $F$ displayed in Figure \ref{fig:unknot}, one computes $\widetilde{\sigma} = -2$, despite the fact that $\partial F$ is the unknot. 

\vspace{0.3cm}
A similar technical difficulty arises when studying the signature associated to the Seifert pairing of non-orientable surfaces (the Goeritz pairing). In \cite{MR500905} Gordon and Litherland found a correction that when added to the signature of the Goeritz pairing leads to a knot invariant, namely the knot signature. Their work is based on the Atiyah-Singer $G$-signature formula \cite{MR236952}, a byproduct of their celebrated index theorem.  In this paper we use the same strategy to find a correction term $e(\Delta,\widetilde{\gamma})$ that, when added to the equivariant signature, leads to a numerical invariant $\widetilde{\sigma}(K)$ independent of the choice of invariant surface. This correction term can be computed starting from any admissible diagram (see Definition \ref{def:admissible}). To do so, we cut the knot at the two fixed points and obtain a pair of arcs $a$ and $b$ which we orient so that they induce opposite orientations on $K$. 

\begin{restatable}{thm}{diagrammaticeuler} \label{thm:diagrammatic_euler_formula}
Let $D$ be an admissible diagram for a (directed\footnote{See Definition \ref{defdirection}.}) strongly invertible knot $K$. Then with respect to the admissible checkerboard surface we have that
\[
e(\Delta, \widetilde{\gamma}) = -\sum_{c \in a\cap b, c \notin h} \epsilon(c),
\]
where the sum is over the crossings of between $a$ and $b$ which do not lie on the axis of symmetry $h$. Here $\epsilon(c)=\pm1$ denotes the sign of the crossing $c$.
\end{restatable}

Note that admissible diagrams of strongly invertible knots always exist. We also show that for an alternating admissible diagram, the correction term $e(\Delta, \widetilde{\gamma})$ is equal to minus the $g$-signature $\widetilde{\sigma}(K)$. 

\begin{restatable}{thm}{alternating} \label{thm:alternating}
Let $D$ be an alternating admissible diagram for a (directed) strongly invertible knot $K$. Then 
\[
\widetilde{\sigma}(K) = \sum_{c \in a\cap b, c \notin h} \epsilon(c),
\]
as above. 
\end{restatable}

In \cite{MR1102258}, Sakuma shows that invertible knots form an infinitely generated group $\mathcal{C}^\text{inv}$ under the equivalence relation of equivariant concordance. At the moment of writing very little is known about this mysterious group, although it seems likely that $\mathcal{C}^\text{inv}$ is non-abelian, and that it contains a copy of $\Z*\Z$.  Analogously to the usual knot signature, we show that $\widetilde{\sigma}(K)$ defines a group homomorphism 
\[
\widetilde{\sigma}: \mathcal{C}^\text{inv}\to \Z \ .
\] 
Unlike the usual knot signature however, the correction term means that $\widetilde{\sigma}(K)$ does not give a lower bound on the equivariant four-genus $\widetilde{g}_4(K)$; see Example \ref{exa:7_4}. Instead it gives a lower bound on the \emph{butterfly 4-genus} $\widetilde{bg}_4(K)$, which is the minimum genus of an invariant surface in $B^4$ with a connected and separating fixed point set. 
\begin{thm} Let $K \subset S^3$ be a directed strongly invertible knot. Then
\[
\widetilde{\sigma}(K) \leq 2\widetilde{bg}_4(K).
\]
\end{thm}

In \cite{MR1102258} strongly invertible knots are investigated by means of the $\eta$-polynomial, a group homomorphism $\eta: \mathcal{C}^\text{inv}\to \Z[t,t^{-1}]$. We show (Proposition \ref{kerneldetection}) that $\widetilde{\sigma}$ can be used to detect elements in the kernel of $\eta$.

Note that to have a well defined connected sum operation on the equivariant knot concordance group $\mathcal{C}^\text{inv}$ it is necessary to decorate strongly invertible knots with a \emph{direction} (or a \emph{framing} in the language of Sakuma). See Definition \ref{defdirection} below. Similarly to the case of the classical knot concordance group, where Livingston famously showed that different orientations may represent different concordance classes, using the invariant $\widetilde{\sigma}$ we show (Proposition \ref{direction}) that in the equivariant case different directions may represent different classes in $\mathcal{C}^\text{inv}$.  Another proof of this fact recently appeared in \cite{BoyleIssa}.

\section{The strongly invertible concordance group}
\subsection{Strongly invertible knots and equivariant surfaces}
\begin{defi}
A knot involution is a pair $(K,\rho)$ where $K \subset S^3$ is a knot, and  $\rho:S^3\to S^3$ is a non-trivial orientation-preserving involution of $S^3$ fixing $K$ setwise. 
\end{defi}

Let $(K,\rho)$ be a knot involution. Since any involution of $S^3$ is conjugate to an involution in $SO(4)$, we have two categories: 
\begin{enumerate}
\item $\rho$ is conjugate to the antipodal map, or
\item  $\rho$ is conjugate to rotation around an unknot.
\end{enumerate}
If we also consider the fact that $\rho$ can act on $K\simeq S^1$ either as $z\mapsto -z$ or $z\mapsto \bar{z}$, there are two possible types of involution in the second category; see Table \ref{table:symmetries}. In this paper we will discuss the case of strongly invertible knots, and in particular their four-dimensional topology.

\begin{defi} A surface $F \subset D^4$ with boundary a strongly invertible knot $K$ is called equivariant if there is an involution $\rho:D^4\to D^4$ restricting to the knot involution on $\partial D^4=S^3$ such that $\rho(F)=F$. 
%Furthermore, we will require that the fixed point set of the restriction $\rho|_F:F\to F$ consists of a single arc. Do we need this? I don't think it belongs in this definition.
\end{defi}

\begin{remk} By Smith theory the fixed point set of $\rho|_F:F \to F$ consists of an arc and a collection of pairwise disjoint circles $\Gamma$ and $|\Gamma| \leq g(F)$. If $F$ is non-orientable there may also be up to $n$ isolated fixed points with $n+2|\Gamma|\leq b_1(F)$.
\end{remk}

It is natural to define the equivariant four-genus of a strongly invertible knot $\widetilde{g}_4(K)$ as the minimum genus of an equivariant orientable surface in $B^4$ bounding $K$.

\begin{table}[t] 
\vspace{0.3cm}
\begin{tabular}{ |c|c|c| } 
 \hline
$\rho$ is conjugated to:  & Antipodal Map & Rotation \\ 
  \hline
$\rho|_K: z\mapsto -z$   &  \begin{tabular}{@{}c@{}}\textbf{freely} \\ \textbf{periodic}\end{tabular} & \textbf{periodic}  \\ 
 \hline
$\rho|_K: z\mapsto \bar{z}$ &\begin{tabular}{@{}c@{}}not \\ possible\end{tabular} &\begin{tabular}{@{}c@{}}\textbf{strongly} \\ \textbf{invertible}\end{tabular} \\ 
 \hline
\end{tabular}

\vspace{0.3cm}\label{fig}
\caption{There are $3$ distinct types of knot involutions preserving the orientation on $S^3$.}
\label{table:symmetries}
\end{table}
 
\subsection{The concordance group of strongly invertible knots}
Two knots $K_0$ and $K_1$ are called concordant if  there exists a smoothly embedded cylinder $C \subset S^3 \times I$ such that $\partial C= C \cap \partial (S^3 \times I ) = K_0 \cup -K_1$. If $(K_0, \rho_0)$ and $(K_1, \rho_1)$ are strongly invertible knots we say that they are equivariantly concordant knots if there exists an involution $\rho: S^3 \times I \to S^3 \times I$ restricting to $\rho_0\cup \rho_1$ on $\partial (S^3 \times I )$ such that $\rho(C)=C$.

We denote by $\mathcal{C}^\text{inv}$ the set of invertible knots modulo the equivalence relation defined by equivariant concordance. In order to have a well defined sum operation on $\mathcal{C}^\text{inv}$ it is necessary to take into account some appropriate decorations. 

\begin{defi}[Directed knots]\label{defdirection} Let $K$ be a strongly invertible knot with axis $A$. Then $A\setminus K=h\cup h'$ consists of two arcs called half-axes. A direction of $K$ consists of a choice of a basepoint on either $h$ or $h'$ and of an orientation of the marked half-axis. \end{defi}

In particular a direction prescribes a preferential basepoint in the fixed set $K\cap A$, namely the terminal endpoint of the marked half-axis. In \cite{MR1102258}, Sakuma refers to directed knots as framed knots; we have chosen to avoid the term \emph{framed} since it has come to be commonly used to refer to a choice of longitude. 

Given two \emph{directed} strongly invertible knots $K_0$ and $K_1$ we define their connect sum as follows.

\begin{defi}
Let $K$ and $K'$ be a pair of directed strongly invertible knots with chosen (oriented) half-axes $h_K$ and $h_{K'}$ respectively. Then their \emph{equivariant connect sum} $K \# K'$ is given by taking the connect sum at the terminal fixed point of $h_K$ and the initial fixed point of $h_{K'}$ so that $K \# K'$ has an obvious strong inversion, see Figure \ref{fig:connect_sum}. The chosen half axis is then $h_{K\#K'} = h_K \cup h_{K'}$.
\end{defi}

With this definition of connected sum the set of \emph{directed} strongly invertible knots (up to equivariant concordance) forms a group  $\mathcal{C}^\text{inv}$. Note that the same equivariant knot with two different directions may lead to two distinct concordance classes, see Proposition \ref{direction} below. 

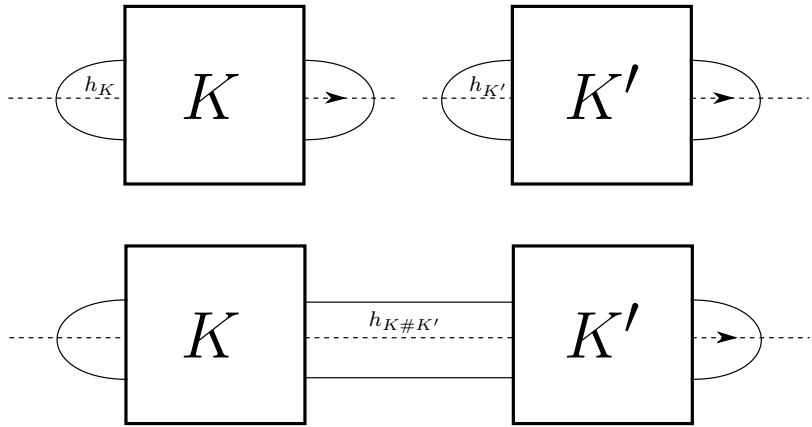
\begin{figure}
\begin{overpic}[width=300pt, grid = false]{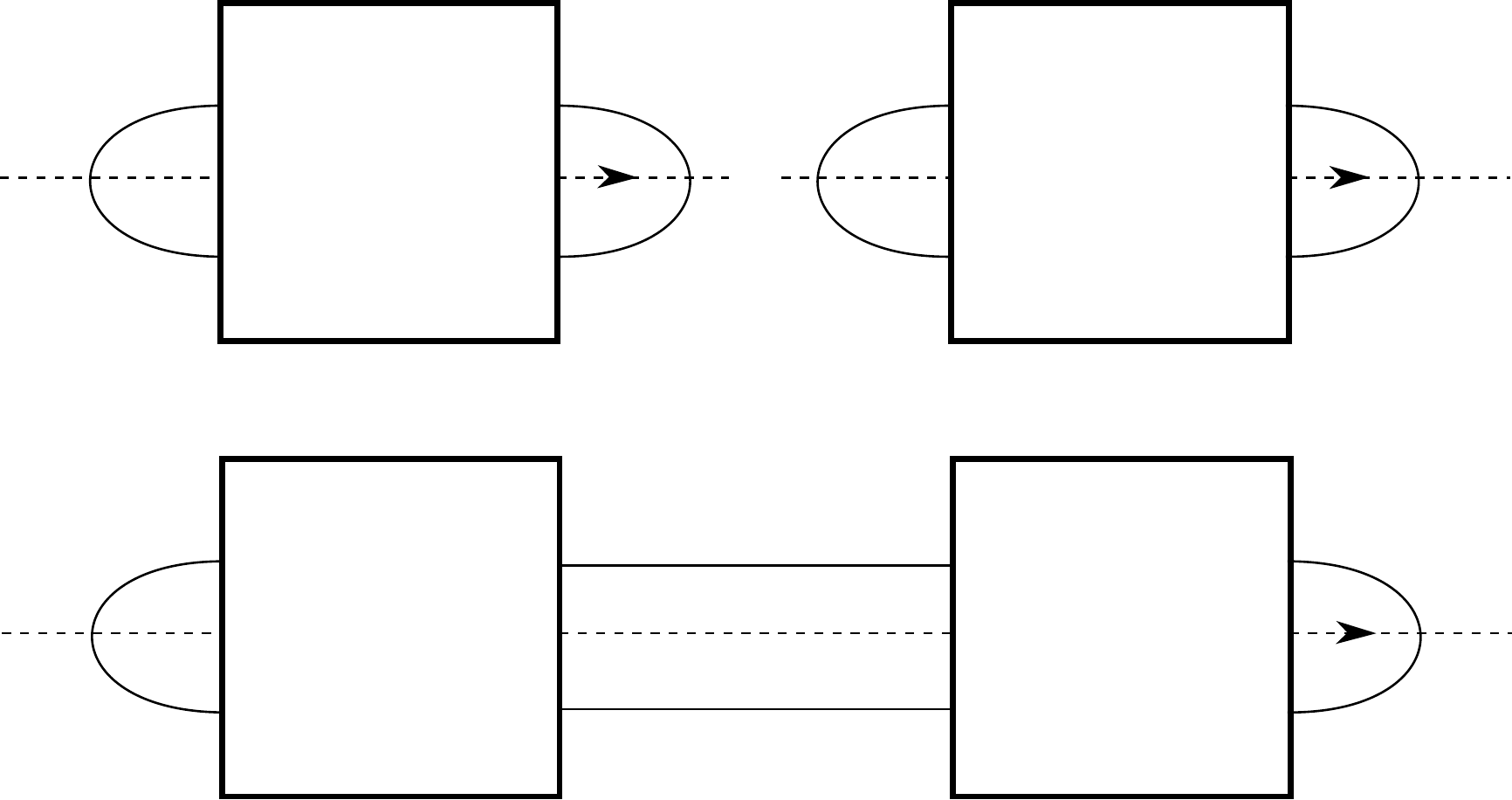}
\put (21.5, 38.5) {\huge{$K$}}
\put (69.3, 38.5) {\huge{$K'$}}
\put (9.5,42) {\tiny{$h_K$}}
\put (57.5,42) {\tiny{$h_{K'}$}}
\put (21.5, 8.5) {\huge{$K$}}
\put (69.3, 8.5) {\huge{$K'$}}
\put (45,12.5) {\tiny{$h_{K \# K'}$}}
\end{overpic}
\caption{The equivariant connect sum $K\#K'$ (bottom) of two strongly invertible knots $K$ and $K'$ (top).}
\label{fig:connect_sum}
\end{figure}

\section{Lifting the strong inversion to branched coverings}

\subsection{Three-dimensional branched coverings} Let $K \subset S^3$ be an oriented knot. In what follows we denote by $ \Sigma^p(K)=\Sigma^p(S^3, K)$ the $p$-fold branched cover of $S^3$ along the knot $K$. The following lemma describes the lifts of the strong inversion to $\Sigma^p(K)$.

\begin{lem} \label{lemma:involution_lifts}
Let $(K,\rho)$ be a strongly invertible knot, $\Sigma^p(K)$ be its $p$-fold cyclic branched cover, $\tau:\Sigma^p(K) \to \Sigma^p(K)$ be a generator of the deck transformations, and $\pi:\Sigma^p(K) \to S^3$ the projection map. Then there are exactly $p$ involutions $\{\widetilde{\rho}_1,\dots,\widetilde{\rho}_p\}$ of $\Sigma^p(K)$ such that $\pi \circ \widetilde{\rho}_i = \rho$. Furthermore, the group generated by $\{\widetilde{\rho}_1,\dots,\widetilde{\rho}_p,\tau \}$ has presentation 
\[
\langle \widetilde{\rho}_1,\tau \mid (\widetilde{\rho}_1)^2 = \tau^p = (\tau \circ \widetilde{\rho}_1)^2 = 1 \rangle
\]
and therefore it is isomorphic to the dihedral group $D_p$. In particular, up to conjugation $\rho$ has a unique lift $\widetilde{\rho}:\Sigma^p(K) \to \Sigma^p(K)$  for $p$ odd, and exactly two lifts $\widetilde{\rho}_1$ and $\widetilde{\rho}_2$  for $p$ even.
\end{lem}
\begin{proof}
Choose a Seifert surface $F \subset S^3$ of $K$. Let $N$ be a neighborhood of $F \subset S^3$ modeled on $F\times I/\sim$ where $(x, t) \sim (x,t')$ if and only if $x\in \partial F$, and $Y=S^3 \setminus N$ be the three-manifold obtained from $S^3$ cutting along $F$. Note that the boundary of $Y$ consists of two parts $R_+$, and $R_-$ with $R_+\simeq F $,  $R_-\simeq -F$, and  $R_+\cap R_-=K$. In particular $\partial Y =D(F)$, the double of $F$. We have that 
\[\Sigma^p(K)= \bigcup_{i \in \Z/p\Z} Y\times i / \sim \]
where $R_+\times i$ is identified with with $R_-\times i+1$. In this model a generator for the deck transformations $\tau: \Sigma^p(K) \to \Sigma^p(K)$ acts as $(x,i)\mapsto(x,i+1)$.

Consider the standard $D_p$-action on the index set $\{1,2,\dots,p\}$. For $g \in D_p$ define $\widetilde{\rho}_g(x, j) = (\rho(x),g(j))$. Then $\{\widetilde{\rho}_1,\dots,\widetilde{\rho}_p\}=\{\widetilde{\rho}_{g_1},\dots,\widetilde{\rho}_{g_p}\}$ where $\{g_1,\dots,g_p\}$ denotes the order two elements of the dihedral group $D_p$. 
 \end{proof}

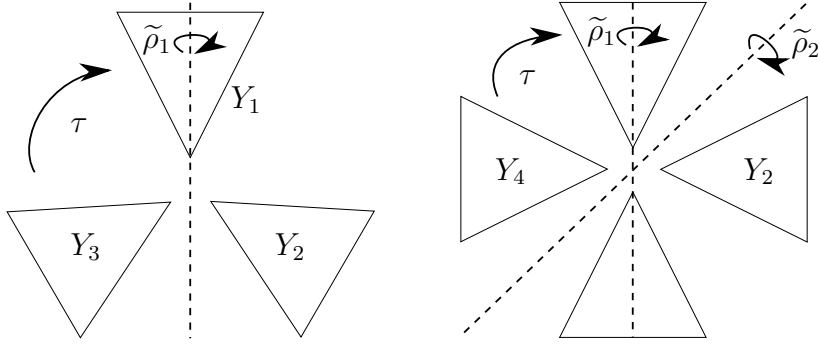
\begin{figure}[t]
\begin{overpic}[width=300pt, grid = false]{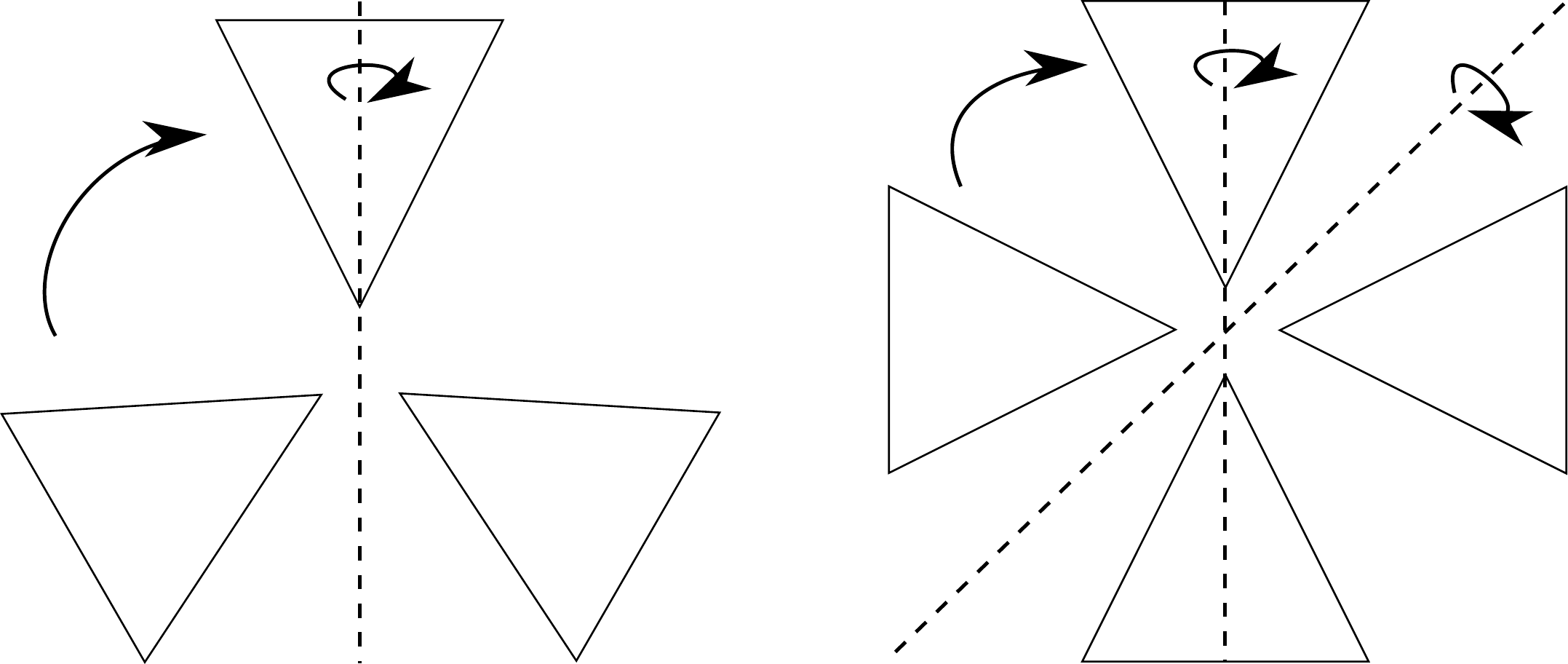}
\put (17, 36.5) {$\widetilde{\rho}_1$}
\put (72.5, 37.5) {$\widetilde{\rho}_1$}
\put (98,36) {$\widetilde{\rho}_2$}
\put (8,26) {$\tau$}
\put (64,32) {$\tau$}
\put (28,29) {$Y_1$}
\put (33.5,10.5) {$Y_2$}
\put (8,10.5) {$Y_3$}
\put (92,20) {$Y_2$}
\put (61,20) {$Y_4$}
\end{overpic}
\caption{On the left is a pair of generators for the $D_3$ symmetry of the three-fold branched cover of a strongly invertible knot. On the right is the four-fold branched cover, and two lifts $\widetilde{\rho}_1$ and $\widetilde{\rho}_2$ which are in different conjugacy classes of $\tau$.}
\label{fig:Involution_lifts}
\end{figure}

%\begin{cor} The lift $\widetilde{\rho}$ has order two. Furthermore  the covering involution $\tau$ and $\widetilde{\rho}$ commute with each other and generate an action of $\Z/p \rtimes \Z/2\simeq D_{2p}$.\QEDB
%\end{cor}

In the case where we fix a direction on a $(K,\rho)$ so that we have a pointed half-axis $h$ and a non-pointed half-axis $h'$, we can distinguish between $\widetilde{\rho}_1$ and $\widetilde{\rho}_2$ in Lemma \ref{lemma:involution_lifts} since one fixes the preimage $\pi^{-1}(h)$ and the other fixes the preimage $\pi^{-1}(h')$. Following \cite{BoyleIssa}, we will designate the lift which fixes $\pi^{-1}(h')$ and acts non-trivially on $\pi^{-1}(h)$ as the \emph{distinguished lift} $\widetilde{\rho}$ of $\rho$.

\subsection{Four-dimensional branched coverings} 
Suppose  $F \subset D^4$ is a properly embedded surface with boundary a strongly invertible knot $K=\partial F$ in $S^3$. Then we can  form the $p$-fold branched cover $\Sigma^p(F, D^4)$ of the four-disk $D^4$ along  the surface $F$. Suppose that $F$ is invariant under an involution $\rho:D^4 \to D^4$ which restricts to a strong inversion on $K$. Then by a similar argument to that for Lemma \ref{lemma:involution_lifts} we get a unique (up to conjugation) lift of $\rho$ to $\Sigma^p(D^4,F)$ for $p$ odd, and a pair of lifts to $\Sigma^p(D^4,F)$ for $p$ even. Specifying a direction on $K$ again determines a unique \emph{distinguished lift} $\widetilde{\rho}:\Sigma^p(D^4,F) \to \Sigma^p(D^4,F)$ for which the fixed set on the boundary is in the preimage of the non-pointed half-axis $h'$.

%If $F$ is the push-off of a Seifert surface in $S^3$ this can be constructed explicitly as follows. Let $M\subset D^4$ be track of an isotopy (relative to $K$) moving $F$ into a Seifert surface in $F' \subset S^3$. Then $M$ is homeomorphic to $F\times I$ with $\partial F \times I$ collapsed to $K=\partial F$. 
%In particular $\partial M=N_+\cup N_-$, where $N_+=F$, $N_-=F'$, and $N_+\cap N_-=K$. 
%Let $D$ denote the result of cutting $D^4$ along $M$. Note that $D  \simeq D^4$, and that $\partial D$ contains two copies $M_+$ and $M_-$ of $M$, with $M_+\cap M_-=F$. Then 
%\[\Sigma^p(F, D^4) = \bigcup_{i \in \Z/p\Z} D\times i / \sim  \ , \]
%where we identify  $M_+ \subset D \times i$ with  $M_- \subset D \times i+1$. (Seep PAC-MAN picture)

%If the knot symmetry $\rho: S^3 \to S^3$ extends to an involution $\rho: D^4 \to D^4$ and $F$ comes from an equivariant Seifert surface then $\rho$ restricts to a map $D\to D$. Thus in our cut-and-paste model of the $p$-fold cover we have an involution defined by $\widetilde{\rho}(x, i)=(\rho(x),i)$. Note that  $\partial\Sigma^p(F, D^4)=\Sigma^p(S^3, K)$, and that $\widetilde{\rho}$ restricts to the lift of the knot involution discussed in the paragraph above.

%If $F\subset D^4$ comes from an admissible Seifert surface $F'\subset S^3$ such as the left-hand checkerboard surface in Figure \ref{fig:7_4Badmissible}, then $\text{Fix}(\rho)\cap D =h'\times I$ and $\bigcup_{i\in \Z/p\Z} h'\times I \times i$ describes a disk $\Delta=\text{Fix}( \widetilde{\rho})$ in $\Sigma^p(F, D^4)$.
 
\section{G-signatures and the Atiyah-Singer formula} 

\subsection{Representation ring}
Let $G$ be a finite group. Recall that given two representations $V$ and $W$ of $G$ we can form the direct-sum representation $V \oplus W$ and the tensor product representation $V \otimes W$. 

Denote by $R(G)$ the representation ring of $G$.  This has one generator $[V]$ for each complex representation $V$ of $G$ and relations:
\[[V\oplus W]=[V]+[W] \  , \ \ \  [V \otimes W]= [V] \cdot[W]  \ .\]
Note that for each $g \in G$ we have a ring homomorphism $R(G)\to \C$. This is the character map $\chi_g$ and is defined by
\[\chi_g[V]= \text{trace}(g:V \to V) \ .\]
If $G$ is compact, for $x\in R(G)$ we have that $x=0$ if and only if $\chi_g(x)=0$ for all $g \in G$.

\subsection{G-signature} Let $M$ be a compact, oriented manifold of even dimension $n=2m$. Then Poincar\' e duality gives rise to an Hermitian  pairing
\[Q_M: H_n(M; \C) \otimes H_n(M; \C) \to \C,\ \ \ Q_M(\alpha \otimes \beta)= \int_M \alpha \wedge \overline{\beta} , \]
for all $\alpha$ and $\beta\in H^n(M; \C)\simeq H_n(M; \C)$. Note that $Q_M$ is symmetric if $m$ is even, and skew-symmetric if $m$ is odd. Suppose that $G$ is a finite group acting on $M$ by orientation-preserving diffeomorphism. Then $\alpha \mapsto g^*\alpha$ defines a representation of $G$ on $V=H_n(M, \C)\simeq \C^{b_n(M)}$ preserving the duality pairing. In this case we can find an orthogonal decomposition 
\[V=V_+ \oplus V_-\oplus V_0\]
with $V_+, V_-$ and $V_0\subset V$ sub-representations such that: $Q_M:V_+\otimes V_+\to \C$ is positive-definite,  
$Q_M:V_-\otimes V_-\to \C$ is negative-definite, and $Q_M(v\otimes w)=0$ for all $v, w\in V_0$.
We define the $G$-signature of $M$ as the formal difference 
\[\text{sign}(M, G)= [V_+]-[V_-] \ .\]
Note that the $G$-signature exists as an element of the representation ring $R(G)$. In. what follows we list some well-known properties of the $G$-signature:
\begin{itemize}
\item $\text{sign}(-M, G)=-\text{sign}(M, G)$
\item $\text{sign}(M\times N, G)=\text{sign}(M, G)\cdot \text{sign}(N, G)$,
\item if $(M,G)=(M_1, G) \cup_\partial (M_2, G)$ then $\text{sign}(M,G)=\text{sign}(M_1, G) +\text{sign}(M_2, G)$
\item if $(M,G)=\partial (W,G)$ for some $W$, then  $\text{sign}(M, G)=0$
\end{itemize}

\subsection{Atiyah-Singer signature formula}In \cite{MR236952} Atiyah and Singer use their celebrated index theorem to compute the character associated to the G-signature:
\[\text{sign}(M, g) = \chi_g (\text{sign}(M, G))= \chi_g[V_+]-\chi_g[V_-]= \text{tr}(g|_{V_+})-\text{tr}(g|_{V_-}) \ .\]

For a finite group $G$ acting by orientation-preserving diffeomorphisms on a closed 4-manifold $M$, the fixed point set
\[
\text{Fix}(G)=\{x\in M \text{ such that } g \cdot x = x \text{ for all } g \in G\}
\]
consists of a set $P$ of isolated fixed points and a collection $F_1, \dots, F_m \subset M$ of pairwise disjoint surfaces. Suppose that $G=\Z/m\Z$ is generated by an orientation-preserving  diffeomorphism $g:M \to M$. Then every fixed point $x\in P$ has a neighborhood $D^2\times D^2$  where $g$ acts as $(z,w) \mapsto (\theta_1 \cdot z, \theta_2 \cdot w)$, with $\theta_1^m=\theta_2^m=1$. Similarly a fixed surface $F_i$ has a neighborhood modeled on a $D^2$-bundle with Euler number $e(F_i)=F_i \cdot F_i$ where $g$ acts on the fibers as rotation of angle $\psi_i=2\pi r/m$, with $\text{gcd}(r,m)=1$. 
In this case (see \cite{MR900251}), we have 

\[
\text{sign}(M, g) =  - \sum_{x\in P} \text{cot}\left( \frac{\theta_1}{2}\right) \cdot \text{cot}\left( \frac{\theta_2}{2}\right) + \sum_{i=1}^m e(F_i) \cdot   \text{cosec}^2 \left( \frac{\psi_i}{2}\right) .
\]

In particular if $P=\emptyset$ and $m=2$ the formula simplifies: 
\[\text{sign}(M, g) =
\cancel{ - \sum_{x\in P} \text{cot}\left( \frac{\theta_1}{2}\right) \cdot \text{cot}\left( \frac{\theta_2}{2}\right) } + \sum_{i=1}^m e(F_i) \cdot   \text{cosec}^2 \left( \frac{\psi_i}{2}\right)
= \sum_{i=1}^m F_i \cdot F_i \]
Also note that if $g$ has order two then $H_2(M;\R)=E_+\oplus E_-$, where $E_+$ and $E_-$ denote the two eigenspaces of $g_*:H_2(M;\R) \to H_2(M;\R)$ corresponding to the eigenvalues $+1$ and $-1$. Furthermore, since $E_+$ and $E_-$ are orthogonal: 
\[\text{sign}(M, g) =  \text{tr}(g|_{V_+})-\text{tr}(g|_{V_-})= \text{sign}(E_+)- \text{sign}(E_-)  \ .\]

\section{Correction terms }
%We would like to apply the $g$-signature to obtain invariants of strongly invertible knots using the distinguished lift of the strong inversion to $\Sigma^n(D^4,F)$ for some invariant surface $F$ with $\partial F = K$. In order to define this $g$-signature, however, we require a longitude of the fixed set.

\subsection{Canonical longitude}\label{longitude} In what follows we will need a choice of longitude for the fixed axis of the distinguished lift $\widetilde{\rho}:\Sigma^2(K) \to \Sigma^2(K)$. Specifically, let $(K, \rho)$ be a \emph{directed} strongly invertible knot and denote by $A$ the axis of $\rho$. The choice of a direction distinguishes the two arcs of $A \setminus K=h \cup h'$, let $h$ be the one containing the base point. In a neighborhood of $h'$ perform a band surgery as suggested by Figure \ref{fig:band_surgery}, so that the resulting two-component link has linking number zero. This specifies a pair of arcs running parallel to $h'$ with endpoints on $K$. 

\begin{figure}
\includegraphics[width=55mm]{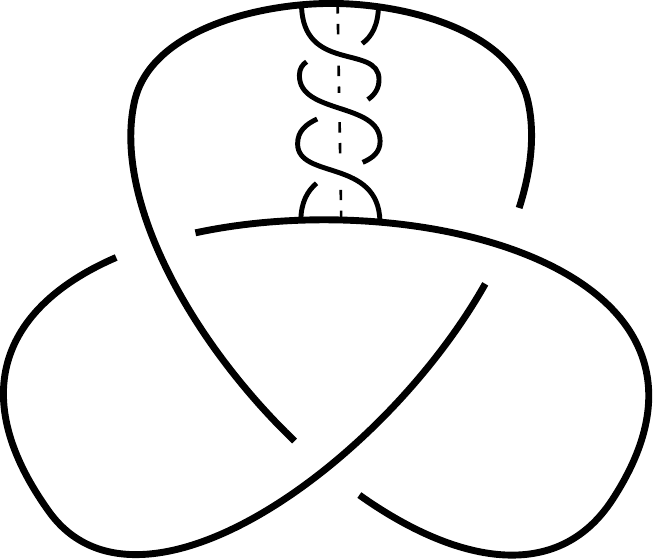}
\caption{A band surgery along one half-axis of the trefoil which produces a 2-component link with linking number 0. Note that in this figure $h'$ is the bounded half-axis and $h$ is the unbounded half-axis.}
\label{fig:band_surgery}
\end{figure}

Choose one of these two arcs and call it $\gamma$. We define the \emph{canonical longitude} $\widetilde{\gamma}$ to be the lift of $\gamma$ to the branched double covering. Note that $\widetilde{\gamma}$ is a longitude of $\widetilde{h'}$. 

\subsection{Euler number} 

Denote by $W = \Sigma^2(D^4,F)$ the branched double cover of $D^4$ along an equivariant surface $F$ with $\partial F = K$, and denote by $\widetilde{\rho}:W \to W$ the distinguished lift of $\rho:(D^4,F) \to (D^4,F)$. Recall that the fixed point set of $\widetilde{\rho} \mid_{\Sigma^2(K)}$ is $\widetilde{h}'$. We then define the equivariant signature of a directed strongly invertible knot $K$ as  follows:
\[\widetilde{\sigma}(K) = \text{sign} \left(W,  \widetilde{\rho} \right) -e(\Delta, \widetilde{\gamma}) \ .\]
where $\Delta \subset W$ denotes the fixed point set of $\widetilde{\rho}$ with $\partial \Delta = \widetilde{h}'$, and $e(\Delta, \widetilde{\gamma})$ denotes the \emph{relative Euler number}, defined as the self-intersection number: 
\[e(\Delta, \widetilde{\gamma})  = \# | \Delta \cap \Delta' | \ .\]
Here $\Delta' \subset W$ denotes a perturbation of $\Delta$ with $\partial \Delta'= \widetilde{\gamma}$, the canonical longitude specified by the direction of $(K,\rho)$ as in Section \ref{longitude}.

\begin{thm} \label{thm:concordance}
The equivariant signature $\widetilde{\sigma}(K)$ does not depend on the choice of the invariant surface $F$. Furthermore, $\widetilde{\sigma}:\mathcal{C}^{inv} \to \mathbb{Z}$ is a group homomorphism from the equivariant concordance group.
\end{thm}
\begin{proof}Suppose that $F_1$ and $F_2$ are two invariant surfaces. Let $W_1$ and $W_2$ denote the branched double covers of $F_1$ and $F_2$ respectively. Let $\Delta_1$ and $\Delta_2$ denote the fixed point sets of the lift of the knot symmetry to $W_1$ and $W_2$ respectively. Furthermore choose $\Delta_1'$ and $\Delta_2'$ perturbations of $\Delta_1$ and $\Delta_2$ so that $\partial \Delta_1= \partial \Delta_2= \widetilde{\gamma}$, the canonical longitude. 

Let $M=W_1 \cup -W_2$ be the manifold obtained gluing $W_1$ and $W_2$ along $\Sigma=\Sigma(K)$. This is a closed $\Z/2\Z$-manifold of dimension four. Inside $M$ the two disks $\Delta_1$ and $\Delta_2$ glue together to form a smooth sphere $S\subset M$ with $\text{Fix}(M,\Z/2\Z)=S$. As consequence of the $G$-signature theorem we have that
\[\text{sign}(M, \Z/2\Z)= S\cdot S \ .\] 
To compute the self-intersection $S\cdot S$ we observe that $S'=\Delta_1'\cup -\Delta_2'$ describes a perturbation of $S$. Consequently, 
\[S\cdot S=\#|S\cap S'|=\# | \Delta_1 \cap \Delta_1' |-\# | \Delta_2 \cap \Delta_2' |=  e(\Delta_1, \widetilde{\gamma})-e(\Delta_2, \widetilde{\gamma}) \ .\] 
On the other hand by Novikov additivity 
\[\text{sign}(M, \Z/2\Z)= \text{sign}(W_1, \Z/2\Z)-\text{sign}(W_2, \Z/2\Z)\ .\]
Hence, $\text{sign}(W_1, \Z/2\Z)-\text{sign}(W_2, \Z/2\Z)= e(\Delta_1, \widetilde{\gamma})-e(\Delta_2, \widetilde{\gamma})$,
from where the identity $\text{sign}(W_1, \Z/2\Z)-e(\Delta_1, \widetilde{\gamma})= \text{sign}(W_2, \Z/2\Z) -e(\Delta_2, \widetilde{\gamma})$ proving that our definition of $\widetilde{\sigma}(K)$ does not depend on the choices made. 
\end{proof}

\section{Computation of the Euler number as  linking number} 
We now give an explicit recipe to compute the Euler number $e(\Delta, \widetilde{\gamma})$ in the case when $W$ is the double branched cover of a sufficiently nice surface.

\subsection{Admissible surfaces} 

\begin{defi} \label{def:admissible}
Let $F \subset S^3$ be a (not-necessarily-orientable) surface with boundary a directed strongly invertible knot $K$ and which is left invariant by the strong inversion. Let $h$ and $h'$ be the half-axes with and without the basepoint respectively. Then $F$ is \emph{admissible} if $h \subset F$ and $h' \cap F = \emptyset$. See Figure \ref{fig:7_4Badmissible} for a pair of examples.
\end{defi}

Note that an admissible surface always exists - take any orientable invariant surface containing $h$, which exists by \cite[Proposition 1]{BoyleIssa} or \cite{Ryota}. Now let $F\subset S^3$ be an admissible surface for a strongly invertible knot $K$. Denote by $\alpha\subset F$ the push-off of $h$ within $F$ with endpoints on $\partial \gamma$. See Figure \ref{fig:canonical_longitude} for an example with the trefoil. 

\begin{lem} \label{lem:relative_euler_linking_formula}
Let $F'\subset D^4$ be an equivariant push-off of $F$ into $D^4$ and let $A = h \cup h'$ be the axis of symmetry. Then with $\gamma, \alpha,$ and $e(\Delta, \widetilde{\gamma})$ as above, we have  
\[e(\Delta, \widetilde{\gamma})=2 \cdot \text{lk}(A,\gamma \cup \alpha)  .\]
\end{lem}
\begin{proof}
Recall that $\Sigma(D^4,F')$ can be constructed as the union of two 4-balls $D_1$ and $D_2$, glued along the thickened surface $F \times [-1,1] \subset \partial D_1 = \partial D_2$ via the identity on $F$ and the reflection involution on $[-1,1]$. The relative Euler number is then the sum of signed intersection points in $D_1$ and  $D_2$. In fact by symmetry, there will be an equal number of signed intersection points in $D_1$ as in $D_2$, so it will be sufficient to count interactions in $D_1$, whence the factor of $2$ in the formula. 

Since $F$ is admissible the arc $\gamma$ intersects $F$ only at its endpoints. Gluing together $\alpha$ and $\gamma$  we get a longitude $\alpha \cup \gamma$ of the axis $A$. Note that $\alpha \subset F \times \{0\}$ so that the two copies of $\alpha \cup \gamma$ in $D_1$ and $D_2$ will glue along $\alpha$. Note that  $\widetilde{\gamma} \subset \Sigma(S^3,K)$ is the union of the two copies of $\gamma$ lying in $D_1$ and $D_2$. In particular, lk$(A, \gamma \cup \alpha)$ will be the intersection number of any disk with boundary $\gamma \cup \alpha$ and the portion of the fixed disk $\Delta$ lying in $D_1$. 
\end{proof}

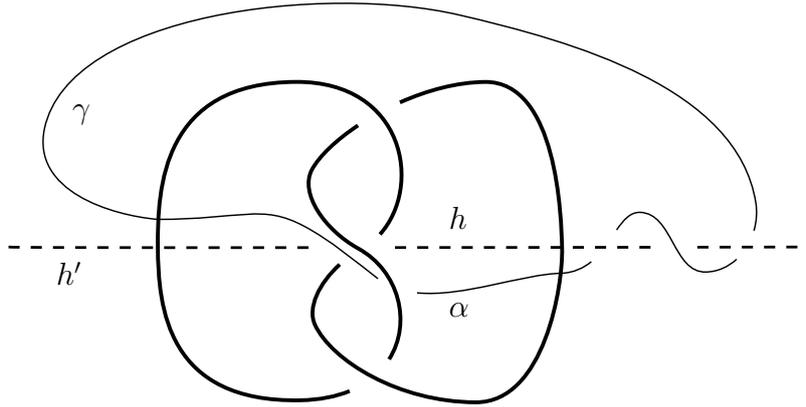
\begin{figure}
\begin{overpic}[width=300pt, grid = false]{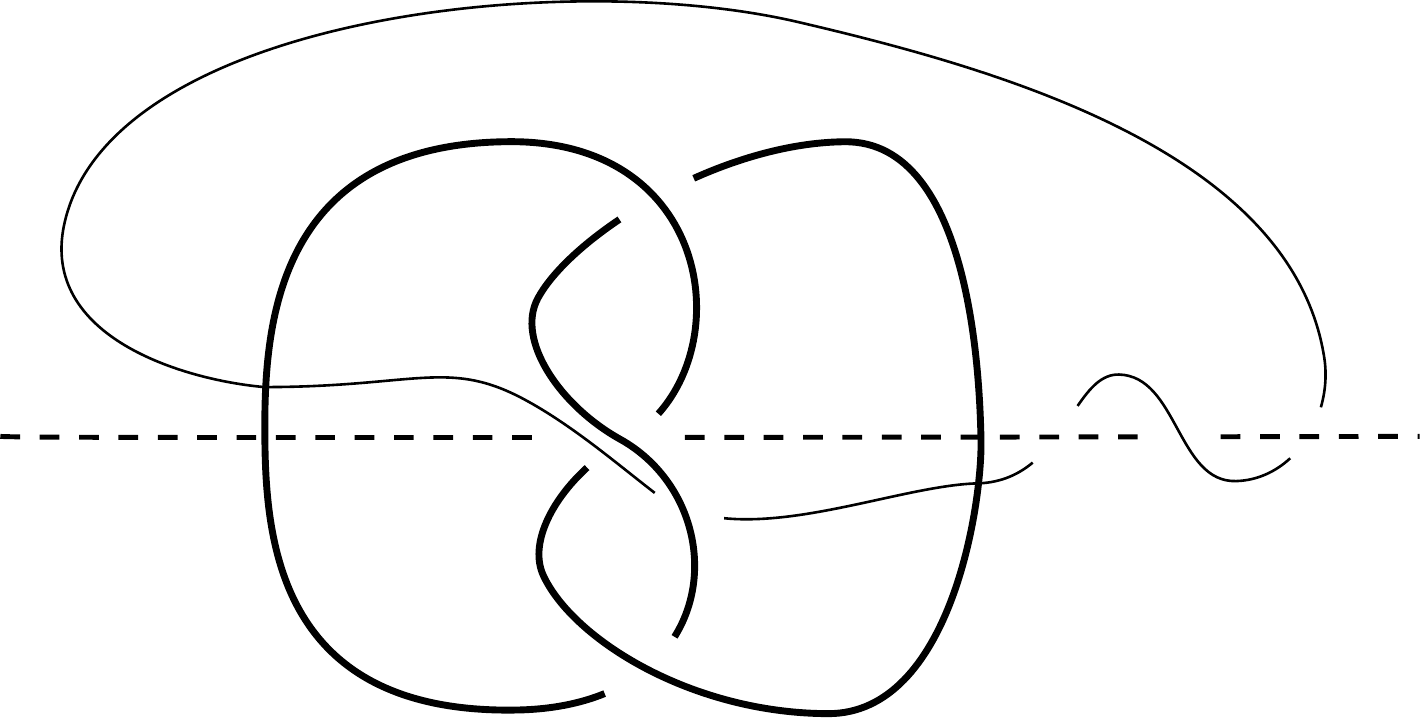}
\put (8, 36) {$\gamma$}
\put (55, 11) {$\alpha$}
\put (6,15) {$h'$}
\put (55,22) {$h$}
\end{overpic}
\caption{The arcs $\gamma$ and $\alpha$ used to build the canonical longitude of the fixed axis in the boundary of the double branched cover of $D^4$ over a push-in of the bounded checkerboard surface.}
\label{fig:canonical_longitude}
\end{figure}

\subsection{Diagrammatic computation of the Euler number} We describe how to compute the relative Euler number combinatorially from certain invariant projections. 

\begin{defi}An \emph{admissible} projection of a directed strongly invertible knot is an invariant knot diagram $D$ with axis of symmetry $A = h \cup h' \subset D$ such that there are no crossings of $D$ along $h'$; see Figure \ref{fig:7_4Badmissible}. 
\end{defi}

Figure \ref{fig:7_4Badmissible} shows an example of an admissible projection. Note that a diagram is admissible if and only if the checkerboard surface containing $h$ is admissible, that is disjoint from $h'$. Let $D$ be an admissible projection. Cutting $K$ at the two fixed points separates $K$ into a pair of arcs $a$ and $b$ with $\overline{a\cup b}=K$. We orient $a$ and $b$ coherently with the half axis $h$ (so that they induce opposite orientations on $K$).
 
\diagrammaticeuler*
\begin{proof}
We will compute $e(\Delta, \widetilde{\gamma})$ using Lemma \ref{lem:relative_euler_linking_formula} by counting the crossings between $\gamma \cup \alpha$ and $A$. It is clear that there is one crossing between $\alpha$ and $A$ for each crossing on $h$, with sign as indicated in Figure \ref{fig:eta_crossing_sign}. 

From the definition of $\gamma$, $a \cup \gamma \cup b \cup \rho(\gamma)$ is a 2-component link with linking number 0. Since $D$ is admissible, there are no crossings between $\gamma \cup \rho(\gamma)$ and $a \cup b$. Thus the signed count of crossings between $a$ and $b$ is the negative of the signed count of crossings between $\gamma$ and $\rho(\gamma)$. Observing that the crossings between $A$ and $\gamma$ correspond exactly with crossings between $\gamma$ and $\rho(\gamma)$, we have that
\[
e(\Delta, \widetilde{\gamma}) = \sum_{c \in h\cap D} \eta(c) - \sum_{c \in a \cap b} \epsilon(c).
\]
However because $a$ and $b$ are oriented so that $\eta(c) = \epsilon(c)$ for crossings on the axis, the first sum cancels exactly with the on-axis crossings in the second sum, giving the desired result.
\end{proof}

\begin{figure}
\begin{overpic}[width=250pt, grid = false]{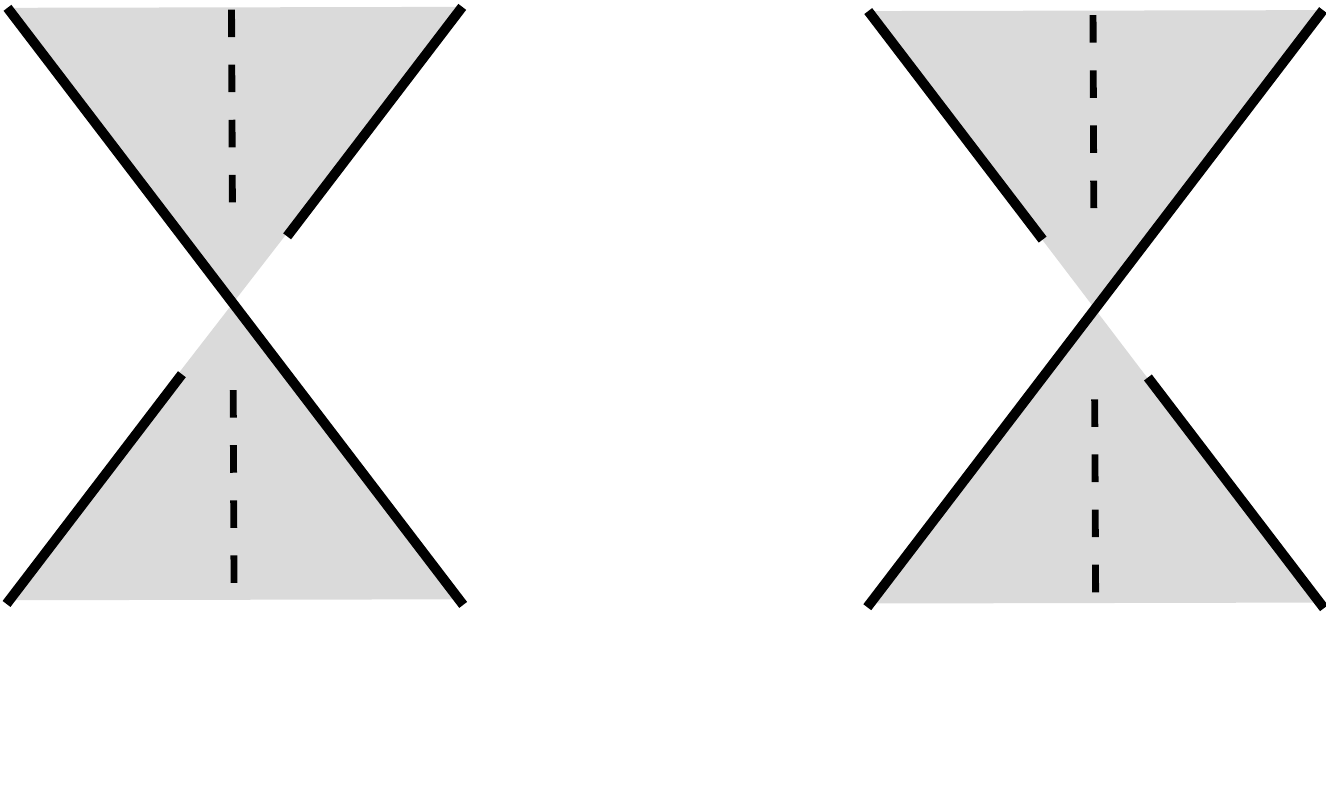}
\put (74, 4) {$\eta(c) = 1$}
\put (8, 4) {$\eta(c) = -1$}
\end{overpic}
\caption{To a crossing $c$ on the half-axis $h$, shown as the dotted line, we associate $\pm 1$ as shown.}
\label{fig:eta_crossing_sign}
\end{figure}

Finally, we give a further simplification when $D$ is an alternating admissible diagram which will facilitate extremely simple computations. In this case, the admissible checkerboard surface has vanishing $g$-signature so that $\widetilde{\sigma}(K) = e(\Delta,\widetilde{\gamma})$.

\alternating*
\begin{proof}
Since $D$ is alternating, the admissible checkerboard surface has a definite linking form. In particular, the half-dimensional eigenspaces $E_+$ and $E_-$ both have maximal (and hence equal) signatures. Thus the $g$-signature vanishes and $\widetilde{\sigma}(K) = -e(\Delta,\widetilde{\gamma})$. Theorem \ref{thm:diagrammatic_euler_formula} then gives the result.
\end{proof}

\section{Examples and applications}

\subsection{Examples}

\begin{exa} \label{exa:7_4}

Consider the directed strongly invertible knot $7_4b^+$ shown as the left-hand diagram in Figure \ref{fig:7_4Badmissible}, and note that this diagram is alternating and admissible so that we can apply Theorem \ref{thm:alternating} to see that $\widetilde{\sigma}(7_4b^+) = -6$. 

On the other hand, consider the right-hand diagram in Figure \ref{fig:7_4Badmissible} and let $F$ be the shown admissible checkerboard surface. We take a basis $\{a,b,c,d,a',b',c',d'\}$ for $H_1(F)$ consisting of the counterclockwise loops around the eight unshaded regions disjoint from the axis. Note that $\rho(x) = -x'$ for $x \in \{a,b,c,d\}$. We can then compute the intersection form on $\Sigma(B^4,F)$ using \cite{MR500905}:
\[
\begin{blockarray}{ccccccccc}
& a & b & c & d & a' & b' & c' & d'\\
\begin{block}{c[cccccccc]}
a  &2 &-1&0 &0 &0 &0 &0 &0 \\
b  &-1& 2&-1&0 &0 &0 &0 &0 \\
c  &0 &-1&2 &-1&0 &0 &0 &0 \\
d  &0 &0 &-1&1 &0 &0 &0 &1 \\
a' &0 &0 &0 &0 &2 &-1&0 &0 \\
b' &0 &0 &0 &0 &-1&2 &-1&0 \\
c' &0 &0 &0 &0 &0 &-1&2 &-1 \\
d' &0 &0 &0 &1 &0 &0 &-1&1 \\
\end{block}
\end{blockarray} .
\]
Next we restrict this intersection form to the $+1$ and $-1$-eigenspaces of the $\rho$ action. Specifically, we have a basis for $E_+$ consisting of elements of the form $x-x'$ for $x \in \{a,b,c,d\}$ and similarly for $E_-$ with elements of the form $x+x'$. We then have the intersection forms
\[
\begin{blockarray}{ccccc}
& a-a' & b-b' & c-c' & d-d' \\
\begin{block}{c[cccc]}
a-a' &4 &-2&0 &0 \\
b-b' &-2& 4&-2&0 \\
c-c' &0 &-2&4 &-2\\
d-d' &0 &0 &-2&0 \\
\end{block} 
\end{blockarray}, \mbox{ and }
\begin{blockarray}{ccccc}
& a+a' & b+b' & c+c' & d+d' \\
\begin{block}{c[cccc]}
a+a' &4 &-2&0 &0 \\
b+b' &-2& 4&-2&0 \\
c+c' &0 &-2&4 &-2\\
d+d' &0 &0 &-2&4 \\
\end{block} 
\end{blockarray}.
\]
so that $\sigma(E_+) = 2$ and $\sigma(E_-) = 4$. Subtracting these, $\sigma(\Sigma(B^4,F)) = -2$, and it remains to compute the relative Euler number $e(\Delta,\widetilde{\gamma})$ using Lemma \ref{lem:relative_euler_linking_formula}. We get 
\[
e(\Delta,\widetilde{\gamma}) = -\sum_{c \in a \cap b, \ c \notin h} \epsilon(c) = 8.
\]
Combining these, $\widetilde{\sigma}(7_4b^-) = \sigma(\Sigma(B^4,F),\widetilde{\rho}) - e(\Delta,\widetilde{\gamma}) = -2 - 8 = -10$. Finally, we note that the left-hand checkerboard surface in Figure \ref{fig:7_4Badmissible} is orientable, so that $\widetilde{g}_4(7_4b) = 1$, and in particular there is no obvious bound on $\widetilde{g}_4$ or $g_4$, coming from $\widetilde{\sigma}$. 
\end{exa}

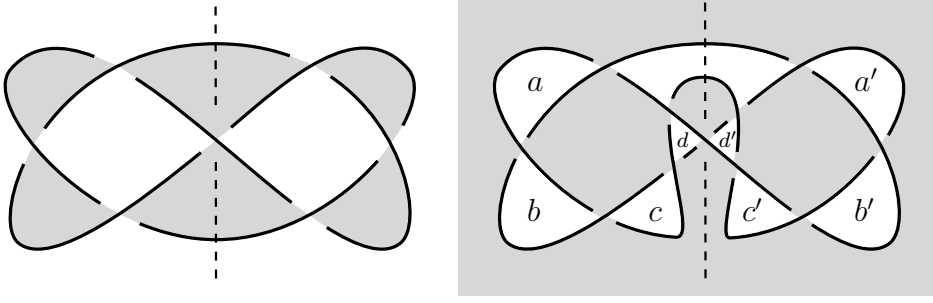
\begin{figure}
\begin{overpic}[width=350pt, grid = false]{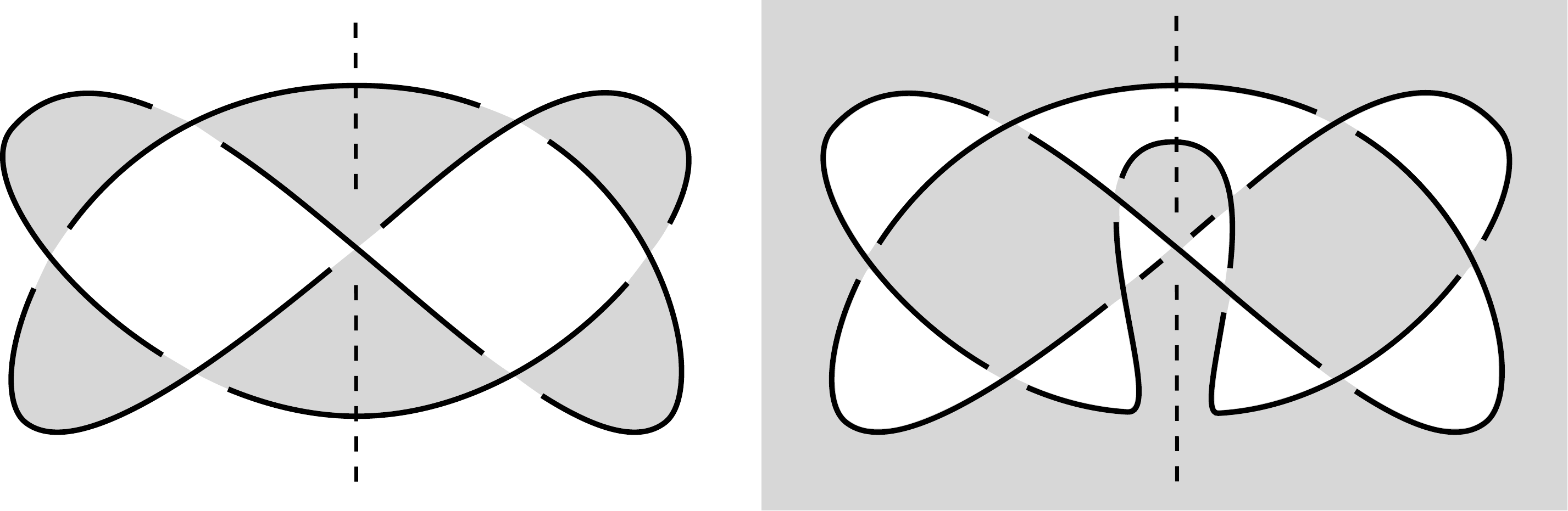}
\put (56,22) {$a$}
\put (56, 8) {$b$}
\put (69, 8) {$c$}
\put (72, 16) {\tiny{$d$}}

\put (91,22) {$a'$}
\put (91, 8) {$b'$}
\put (79, 8) {$c'$}
\put (76.5, 16) {\tiny{$d'$}}

\end{overpic}
\caption{A pair of admissible diagrams for different directions $7_4b^+$ (left) and $7_4b^-$ (right) for a strong inversion on $7_4$. Note in each that the half axis which is not contained in the shaded checkerboard surface is disjoint from that surface.}
\label{fig:7_4Badmissible}
\end{figure}

The following propositions immediately follow from this example (the first also follows from \cite{BoyleIssa}).

\begin{prop}\label{kerneldetection}
There are knots with Sakuma polynomial $\eta(K) = 0$ which are not equivariantly slice.
\end{prop}
\begin{proof}
The directed strongly invertible knot $K = 7_4b^+ \# m7_4b^-$ has $\widetilde{\sigma}(K) = -6 - (-10) = 4$, so that $K$ is not equivariantly slice by Theorem \ref{thm:concordance} even though the Sakuma polynomial $\eta(K) = 0$.
\end{proof}

\begin{prop}\label{direction}
The invariant $\widetilde{\sigma}(K)$ can often distinguish between the two directions on a strongly invertible knot $K$. 
\end{prop}

\begin{exa}
The g-signature is also readily computed for many infinite families. For example, consider the alternating torus knot $T(2,2n+1)$ for which an admissible diagram is shown in Figure \ref{fig:alt_torus_knots}. By Theorem \ref{thm:alternating}, we then have that $\widetilde{\sigma}(K) = -2n$, with a contribution of $-1$ from each off-axis crossing. 

\begin{figure}
\begin{overpic}[width=250pt, grid = false]{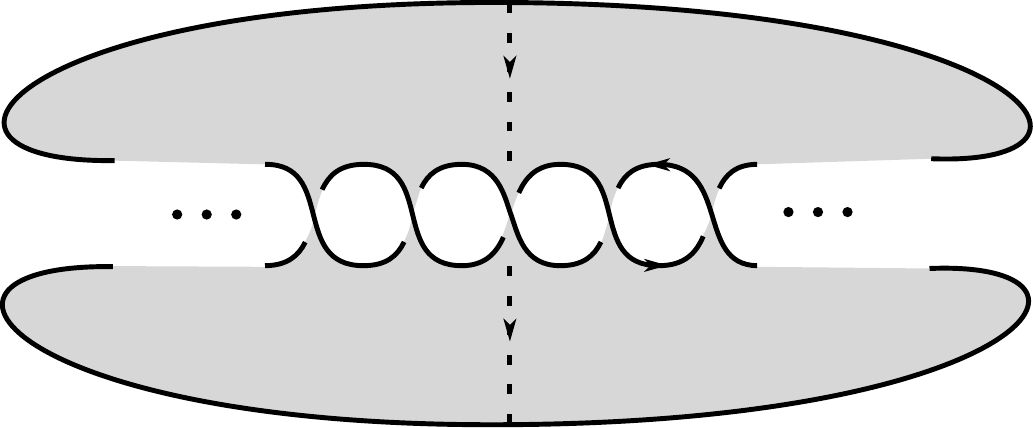}
\put (52.5, 19.5) {$a_1$}
\put (62, 19.5) {$a_2$}
\put (43, 19.5) {$\overline{a}_1$}
\put (33.5, 19.5) {$\overline{a}_2$}
\end{overpic}
\caption{An admissible diagram for the torus knot $T(2,2n+1)$.}
\label{fig:alt_torus_knots}
\end{figure}
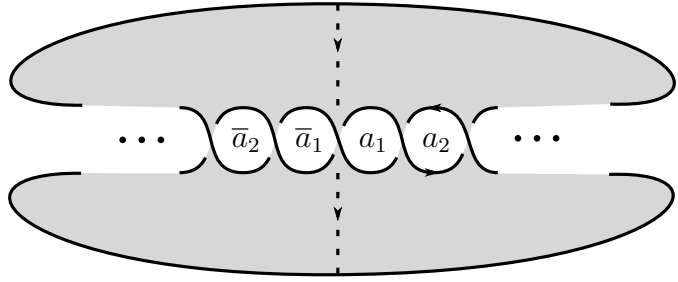
\end{exa}

\subsection{Butterfly genus}
In \cite{BoyleIssa} the authors introduced the notion of a butterfly surface, and the corresponding notion of knot genus (butterfly genus). 

\begin{defi} Let $(K, \rho)$ be a strongly invertible knot. A surface $F\subset D^4$ with $\partial F= K \subset S^3$ is called a butterfly surface if there is a smooth extension $\rho:D^4 \to D^4$ of the involution with $\text{Fix}(\rho)$ intersecting the surface in a separating arc.
\end{defi}

A strongly invertible knot may not bound a butterfly surface; see \cite[Theorem 6]{BoyleIssa}. 

\begin{thm}\label{thm:butterfly}
Let $F\subset D^4$ be a butterfly surface with boundary a strongly invertible knot $(K,\rho) \subset S^3$. Then 
\[\widetilde{\sigma}(K) = \sigma(\Sigma(D^4,F), \widetilde{\rho}) \ , \]
that is the relative Euler number vanishes.  
\end{thm}
\begin{proof}Let $D\subset D^4$ denote the fixed point set of $\rho$. Then $D \cap F$ is an arc $\beta$. Denote by $D'$ the component of $D \setminus \beta$ containing in its boundary the portion $h'$ of the axis that does not contain the basepoint. Let $F'$ be the result of performing a finger move to $F$ along $D'$; see Figure \ref{fig:butterfly_proof}. Then the portion of $F'$ emerging on the boundary $S^3$ is a band $B$. Note that since $\beta$ disconnects $F$, band surgery on $K$ along $B$ produces a two component link with linking number zero.

The difference $F' \setminus (F \cup B)$ consists of two disks $D_+$ and $D_-$ parallel to $D'$. The preimage of $D_+$ in the branched double cover $\Sigma(D^4,F)$ is the graph of a never-vanishing section $s$ of the normal bundle of $\Delta=\text{Fix}(\widetilde{\rho})$. Since  $s(\partial \Delta)=\widetilde{\gamma}$ is the canonical longitude then the relative Euler number $e(\nu\Delta, \widetilde{\gamma})$ vanishes.
\end{proof}

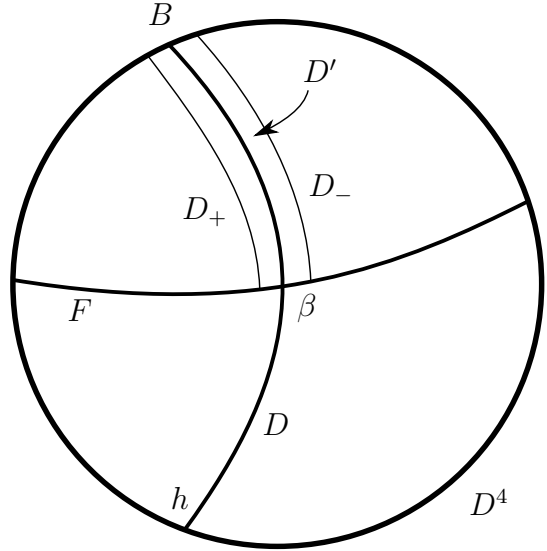
\begin{figure}
\begin{overpic}[width=200pt, grid = false]{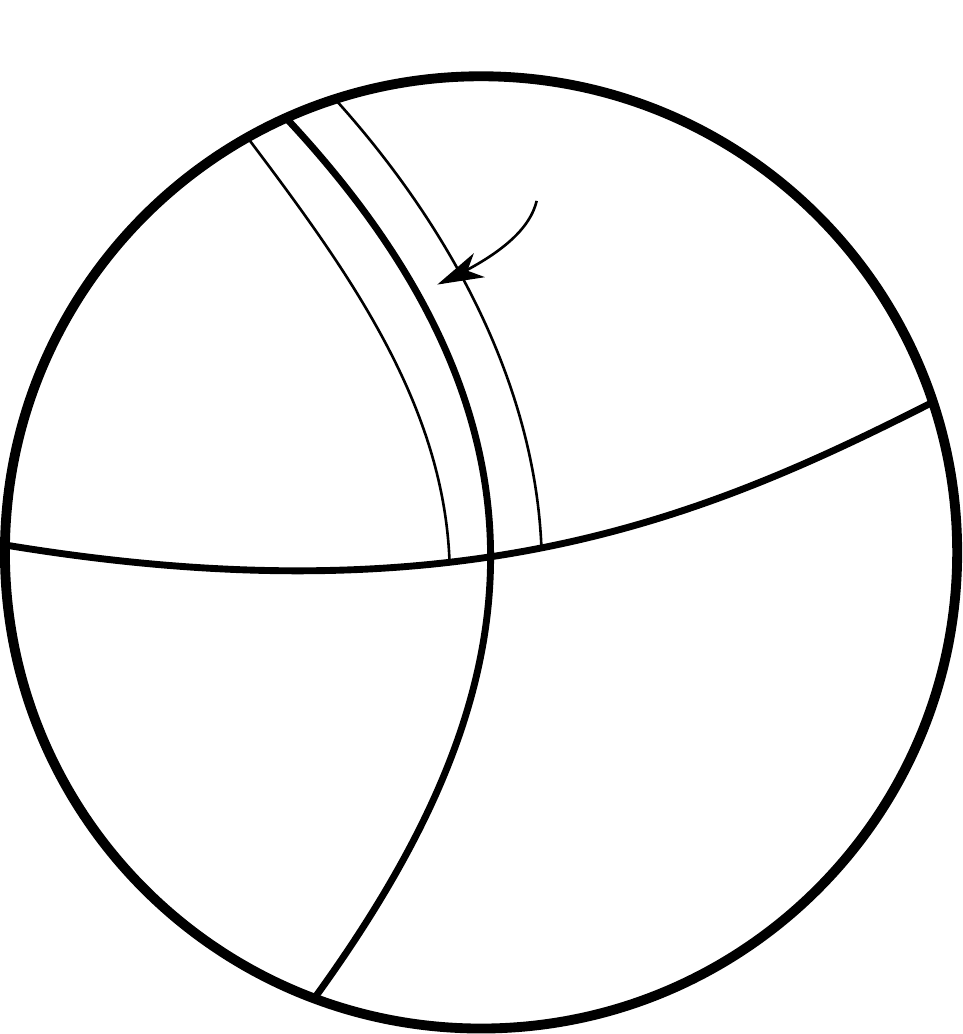}
\put (80, 6) {$D^4$}
\put (30, 58) {$D_+$}
\put (52, 62) {$D_-$}
\put (50, 41) {$\beta$}
\put (24, 92) {$B$}
\put (44, 20) {$D$}
\put (10, 40) {$F$}
\put (28, 7) {$h$}
\put (51,82) {$D'$}
\end{overpic}
\caption{A schematic for the surface $F'$ obtained by a finger move on $F$ along $D'$ in the proof of Theorem \ref{thm:butterfly}.}
\label{fig:butterfly_proof}
\end{figure}
\begin{remk}
Note that our $g$-signature agrees with the one defined in \cite{BoyleIssa} when the knot bounds a butterfly surface (and was not defined in \cite{BoyleIssa} otherwise). 
\end{remk}
\begin{cor}
If $(K,\rho)$ is a strongly invertible knot  then
\[\frac{1}{2}\left|\widetilde{\sigma}(K)\right|\leq \text{bg}_4(K, \rho) \ .\]
In particular the equivariant signature of an equivariantly slice strongly invertible knot vanishes.
\end{cor}
\begin{proof}If $F$ is a butterfly surface then $\widetilde{\sigma}(K)=  \sigma(\Sigma(D^4,F), \widetilde{\rho})$ and 
\[- b_1(F) \leq  \sigma(\Sigma(D^4,F), \widetilde{\rho}) \leq b_1(F)  \ .\]
Thus $\left|\widetilde{\sigma}(K)\right|\leq b_1(F) = 2 g(F) \leq 2 \cdot \text{bg}_4(K, \rho)$.
\end{proof}

\subsection{Final remarks} We also considered the possibility of defining invariants from higher order branched coverings. In the case of odd coverings these invariants seems to vanish. The invariants from even coverings on the other hand only depend on $\widetilde{\sigma}$ and the Tristram-Levine signatures.

\begin{CJK}{UTF8}{min}
\bibliography{Bibliography}
\bibliographystyle{alpha}
\end{CJK}
\end{document}